\documentclass[12pt]{amsart}
\usepackage{amsfonts}
\usepackage{amscd}
\usepackage{amssymb}
\usepackage{graphics}

\usepackage{amsmath}

\usepackage{hyperref}
\hypersetup{colorlinks,citecolor=blue}

\newtheorem{theorem}{Theorem}

\newtheorem{definition}[theorem]{Definition}

\newcommand{\BZ}{{\mathbb{Z}}}
\newcommand{\BN}{{\mathbb{N}}}

\newcommand{\BC}{{\mathbb{C}}}

\newcommand{\act}{\curvearrowright}

\newcommand{\gD}{\Delta}
\newcommand{\gd}{\delta}
\newcommand{\gb}{\beta}

\newcommand{\gC}{\Gamma}
\newcommand{\gc}{\gamma}

\newcommand{\gS}{\Sigma}
\newcommand{\gO}{\Omega}

\newcommand{\ga}{\alpha}

\newcommand{\GL}{\text{GL}}

\newcommand{\SO}{\text{SO}}

\newtheorem{prop}{Proposition}[section]
\newtheorem{thm}[prop]{Theorem}
\newtheorem{lem}[prop]{Lemma}
\newtheorem{cor}[prop]{Corollary}

\theoremstyle{definition}

\newtheorem{rem}[prop]{Remark}

\catcode`\@=11
\long\def\@savemarbox#1#2{\global\setbox#1\vtop{\hsize\marginparwidth 
  \@parboxrestore\tiny\raggedright #2}}
\catcode`\@=12

\begin{document}
\author{Tsachik Gelander}

\address{Mathematics and Computer Science\\
Weizmann Institute\\
Rechovot 76100, Israel\\}
\email{tsachik.gelander@gmail.com}

\date{\today}


\title{Convergence groups are not invariably generated}

\maketitle

\begin{abstract}
It was conjectured in \cite{KLS} that non-elementary word hyperbolic groups are never invariably generated. We show that this is indeed the case even for the much larger class of convergence groups. 
\end{abstract}

\section{introduction}

J. Wiegold \cite{W76} showed that non-abelian free groups are not invariably generated. Kantor, Lubotzky and Shalev \cite{KLS} conjectured (in the form of a question) that every non-elementary Gromov hyperbolic group is non-invariably-generated. We show that this is indeed the case for the much wider class of convergence groups. We prove that every non-elementary torsion-free convergence group admits a proper (infinite rank) free subgroup that meets all the conjugacy classes. For convergence groups with torsion, we obtain a similar result, where the free subgroup is replaced by (an infinite) free product of cyclic groups, and we sometimes have to mod out by a finite normal subgroup.


\subsection{Invariable generation} 
Recall that a subset $S$ of a group $G$ {\it invariably generates} $G$ if $G= \langle s^{g(s)} | s \in S\rangle$ for every choice of $g(s) \in G,s \in S$. One says that a group $G$ is {\it invariably generated}, or shortly IG, if such $S$ exists, or equivalently if $S=G$ invariably generates $G$. 
Equivalently, $G$ is IG iff no proper subgroup meets every conjugacy class, or in other words, iff every transitive permutation representation on a non-singleton set admits a fixed-point-free element. 

A well known simple counting argument shows that every finite group is IG. Obviously, abelian groups are IG. More generally J. Wiegold \cite{W76} showed that the class of IG groups is closed under extensions, hence contains all virtually solvable groups. Clearly this class is also closed to quotients.
On the other hand the class IG is not closed under direct union; for instance the group of finitely supported permutations of $\BN$ is clearly not IG, since every element admits a conjugate whose support is contained in $2\BN$. Moreover, Wiegold \cite{W77} gave an example of an IG group whose commutator is not IG, proving in particular that the class IG is not subgroup closed. 

Let us list some classical examples of non-IG groups:
\begin{itemize}
\item Groups with a single non-trivial conjugacy class are non-IG (cf. \cite{HNN,Osin}). 

\item $G=\SO(3)$. To see that consider the standard rotation action of $G$ on the sphere $S^2$, and recall that every element fixes some point (and its antipode). 
More generally, every nonabelian connected compact group is non-IG, since every abelian subgroup is contained in a maximal one and all maximal abelian subgroups are conjugated.

\item $G=\GL_n(\BC),~n\ge 2$. Indeed every matrix can be conjugated to the Borel subgroup $B$ of upper triangular matrices. More generally, by \cite[Proposition 2.4]{KLS} a non-virtually-solvable linear algebraic group over an algebraically closed field is never IG.  
\end{itemize}

%
%
%
\medskip


In \cite{W76}, Wiegold proved that the free group on two (or more) letters $F_{\{ a,b,\ldots\}}$ is non-IG by producing a list $L$ of conjugacy class representatives which are jointly independent. To recall his construction, let $\{ w_n\}$ be conjugacy class representatives which start and end with a non-zero power of $b$, then take $L=\{w_n^{a^n}:n\in\BN\}$. 

W. M. Kantor, A. Lubotzky and A. Shalev asked in \cite{KLS} (among many other questions) whether every non-elementary Gromov hyperbolic group is non-IG. In this note we treat this problem for the much wider class of convergence groups.

\subsection{Convergence groups} \label{sec:convergence}
Convergence groups were introduced by H. Furstenberg in \cite{Fur} under the name ``Dynkin groups". 
The term {\it convergence groups} 
was given by Gehring and Martin in \cite{GM:Convergence} (who studied properties of
Kleinian groups through their action on the boundary, thus restricted to
convergence group actions on spheres).
These are groups admitting an action on a compact metrizable space $K$ such that the corresponding  action on triples $\{(x_1,x_2,x_3)\in K^3: i\ne j\Rightarrow x_i\ne x_j\}$ is proper (see also the equivalent Definition \ref{defn:convergence} below).
This class includes
Gromov hyperbolic groups
\cite{Tukia:convergence, Tukia:Uniform_convergence_groups, Bowditch:characterisation_hyperbolic}, relatively
hyperbolic groups \cite{Yaman:Relatively_Hyp_Convergence}, as well as any group acting properly discontinuously on a complete locally compact Gromov hyperbolic space, and in particular every non-elementary discrete subgroup of a rank one simple Lie group. Obviously, this class is subgroup closed.  

\medskip


The main purpose of this paper is to prove:

\begin{thm}\label{thm}
Let $\gC$ be a non-elementary convergence group admitting a faithful minimal convergence action. Then $\gC$ has an independent set $I$ consisting of one representative of every non-trivial conjugacy class, which generates a proper subgroup $\langle I\rangle\lneqq \gC$.
\end{thm}  

By definition, a set $I$ in a group is {\it independent} if the subgroup $\langle I\rangle$ is the free product of the cyclic groups $\langle\gc\rangle,~\gc\in I$.

The kernel $\ker(\act)$ of a convergence action  $\gC\act X$ is always finite, and in many natural examples it is trivial. Theorem \ref{thm} could be stated without the faithfulness assumption, in which case $I$ would consist of representatives of the conjugacy classes not belonging to the finite group $\ker(\act)$. 

\begin{cor}
Non-elementary convergence group are non-IG.
\end{cor}


\section{Definitions and useful properties of convergence groups}
We refer to \cite{Bowditch:Convergence} for supplementary background on convergence groups.

\begin{definition}
\label{defn:collapsing} An infinite set $\Phi$ of homeomorphisms
of a compact topological space $X$ is called {\it collapsing} with respect to a
pair of (not necessarily distinct) points $(a,r)$ if for every
pair of compact sets $K \subset X \setminus \{r\}$ and $L \subset
X \setminus \{a\}$, the set $\{\phi \in \Phi : \phi\cdot K \cap L \ne
\emptyset \}$ is finite. We shall then call $a$ the attracting
point and $r$ the repelling point of $\Phi$. A set $\Phi$ is {\it
collapsing} if it is collapsing with respect to some pair of
points.
\end{definition}

\begin{definition}
\label{defn:convergence} An action of a group $\Gamma$ on a compact Hausdorff space $X$ is
said to have the convergence property if every infinite subset $\Phi < \Gamma$ contains an
infinite subset $\Phi' \subset \Phi$ which is collapsing. A {\it convergence group} is a group
that admits a convergence action on some compact metrisable space consisting of more than two points.
\end{definition}

It follows easily from the definition that the kernel of every convergence action is finite. In particular a convergence group with no
finite normal subgroups admits a faithful convergence action. 
It also follows from the definition that every element of infinite order in $\gC$ fixes
either one or two points of $X$. As a Corollary we obtain the ``usual'' classification of elements
into three mutually exclusive categories: elements of finite order are called {\it elliptic},
elements of infinite order fixing exactly one point are called {\it parabolic} and elements of
infinite order fixing two points of $X$ are called {\it loxodromic}. 
We may denote the fixed points of a non-torsion element $g$ by $g^+$ and $g^-$, and think of them as the future limit and past limit, so that the following is satisfied:
\begin{itemize}
\item $\forall x\ne g^-$ we have $g^n\cdot x\to g^+$,
\item $g^+=g^-$ iff $g$ is parabolic.
\end{itemize} 
Moreover, given any pair of neighbourhoods $g^-\in R,g^+\in A$ there is $n_0$ such that for all $n\ge n_0$, $g^n\cdot (X\setminus R)\subset A$.
For $x\ne g^-$ and a set $U\subset X$ containing $g^+$ in its interior, we denote by $N(g,U,x)\in\BN$ the minimal integer such that
$$
 n\ge N(g,U,x)\Rightarrow g^n\cdot x\in U.
$$

A convergence group is called elementary if
it is finite or if it stabilizes a nonempty subset of $X$ with at most $2$ elements.

The {\it limit set} $L(\gC)\subset X$ is characterised as the set of all limit points $\lim \gc_n\cdot x$, where $x\in X$ is constant and $(\gc_n)$ is a sequence of distinct elements in $\gC$, or as the closure of the set $\{ g^+:g\in\gC~\text{non-torsion}\}$, assuming $\gC$ is non-torsion, which is usually the case (see Lemma \ref{lem:lox}). The assumption that $\gC$ is non-elementary is equivalent to $\text{Card}(L(\gC))>2$ in which case $L(\gC)$ is the unique minimal $\gC$-invariant compact set and it is perfect, i.e. has no isolated points. The minimality of $\gC\act L(\gC)$ means that every orbit is dense. Replacing $X$ by $L(\gC)$ we may always suppose that the action of $\gC$ on $X$ is minimal.

%
%
%

The following two lemmas are well known and elementary:

\begin{lem}\label{lem:lox} 
A non-elementary convergence
group contains a loxodromic element. Moreover the set of pairs $\{ (g^-,g^+)\in X^2:g~\text{is loxodromic}\}$ is dense in $L(\gC)^2$.  
\end{lem}

\begin{proof}[Sketched proof]
If $\Phi\subset \gC$ is a collapsing sequence with attracting and repelling points $a,r\in L(\gC)$, we may, up to multiplying $\Phi$ by some $\gc_0\in\gC$, suppose that $a\ne r$. Let $O_a,O_r$ be disjoint non-complementary open neighbourhoods of $a,r$ respectively, and pick $\phi\in\Phi$ with $\phi\cdot (X\setminus O_r)\subset O_a$. This forces $\phi$ to be loxodromic, proving the existence statement. 

To prove the density statement, let $\phi$ be an arbitrary loxodromic element and let $U,V\subset L(\gC)$ be disjoint non-complementary open sets. Since the action $\gC\act L(\gC)$ is minimal, there are $\ga,\gb\in\gC$ such that $\ga\cdot \phi^+\in U$ and $\gb\cdot \phi^-\in V$. Pick $n\in\BN$ sufficiently large so that $\phi^n\cdot (X\setminus \gb^{-1}V)\subset \ga^{-1}\cdot U$. Then we have
$\ga \phi^n \gb^{-1} \cdot (X\setminus V)\subset U$. This implies that $g=\ga \phi^n \gb^{-1}$ is loxodromic with $(g^-,g^+)\in V\times U$.
\end{proof}


\begin{lem}\label{lem:placement}
Let $\gC\act X$ be a minimal non-elementary convergence action. Given a proper closed subset $\gS\subset X$ and an open subset $O\subset X$ there is an element $\gd\in\gC$ for which $\gd\cdot\gS\subset O$.  
\end{lem}

\begin{proof}
By Lemma \ref{lem:lox} there is a loxodromic element $g\in \gC$ with $g^+\in O$ and $g^-\notin \gS$. Thus for a sufficiently large $n$ we have $g^n\cdot \gS\subset O$.
\end{proof}

%

%

We will make use of the fact that:

\begin{prop}\label{lem:countable}
A convergence group is countable.
\end{prop}

\begin{proof}
Let $\gC\act X$ be a convergence action on a metrisable compact space $X$. Pick three distinct points $a,b,c\in X$ and a countable dense set $D\subset X$. Suppose by way of contradiction that $\gC$ is uncountable. Since for every $g\in\gC$, $g\cdot D$ is dense, and $X$ is second countable, a pigeon hall type argument implies that for some $x,y,z\in D$, the set 
$$
 \{ g\in\gC: g\cdot x\in U_a, g\cdot y\in U_b,g\cdot z\in U_c\}
$$ 
is uncountable and in particular non-empty for any triple of open sets $a\in U_a,b\in U_b,c\in U_c$. Thus we can produce a sequence of distinct elements $g_n\in \gC$ such that $g_n\cdot (x,y,z)\to (a,b,c)$ in $X^3$, in contrast to the convergence property.
\end{proof}

The following proposition is essential for our purpose: 

\begin{prop}[Main Proposition]\label{prop:main}
Let $\gC\act X$ be a minimal non-elementary convergence action. Then for every non-torsion element $\gc\in\gC$ there are two sets $\gO^+,\gO^-$ such that $\gc^n\cdot (X\setminus \gO^-)\subset \gO^+,~\forall n\in\BN$ and the complement $X\setminus (\gO^-\cup \gO^+)$ has a non-empty interior.
\end{prop}

\begin{proof}
Let $A$ be a neighbourhood of $\gc^+$ which is not dense in $X$.
Set 
$$
 A_n:=\{ x\in X:N(\gc,\overline{A},x)\le n\}.
$$ 
Then 
\begin{itemize}
\item the $A_n$ are closed,
\item $A_n\subset A_{n+1}$,
\item $\cup_n A_n\supset X\setminus\{ \gc^-\}$, and
\item $A_{n+1}=\gc^{-1}\cdot A_n$.
\end{itemize}
It follows from the Baire category theorem that there is a minimal $n\in\BN$ such that $A_n\setminus A_1$ has a non-empty interior. 
Then $A_n\setminus A_{n-1}$ has a non-empty interior. However, as the sets $A_i\setminus A_{i-1}$ are all homeomorphic, we deduce that $n=2$.
Thus we may take 
$$
 \gO^+=A_1~\text{and}~\gO^-=X\setminus A_2.
$$
\end{proof}

Let us say that a group action $\gC\act X$ on a topological space is {\it generically free} if the fixed point set of any non-trivial element (an element not belonging to the kernel) is nowhere-dense. As Shahar Mozes pointed out to me, it is easy to construct an example of a non-elementary discrete subgroup of $\text{Aut}(T)$, where $T$ is a $4$-regular tree with a non-trivial elliptic element that fixes an open set in the boundary $\partial(T)$. The next proposition shows that this cannot happen for minimal convergence actions:

\begin{prop}
A non-elementary minimal convergence action is generically free.
\end{prop}

\begin{proof}
We will argue by way of contradiction.
Suppose that $\gC$ admits a minimal non-elementary convergence action $\gC\act X$ which is not generically free, and let $\gc\in\gC$ be an element whose fixed point set $A=\text{Fix}(\gc)$ is a proper closed subset with a non-empty interior $A^\circ$. 
In view of Lemma \ref{lem:placement} there is an element $\gd\in\gC$ such that $\gd^{-1}\cdot A\varsubsetneq A^\circ$. This implies that $\gd^{n-1}\cdot A\varsubsetneq \gd^{n}\cdot A$, for every $n\in \BN$. Let $\gc_n=\gc_{n-1}^\gd=\gc^{\gd^n}$, then $\text{Fix}(\gc_n)=\gd^n\cdot A$ form a strictly increasing sequence of closed sets, and in particular the set $\{\gc_n:n\in\BN\}$ is infinite.  
On the other hand, since $X$ is perfect, the non-empty open set $A^\circ$ contains $3$ distinct points, and as the action is convergence, its pointwise stabiliser is finite. A contradiction.
\end{proof}

In view of Baire's category theorem we deduce:

\begin{cor}\label{cor:wandering}
Let $\gC\act X$ be a non-elementary minimal convergence action, then there is a point $x\in X$ with a trivial stabiliser.
\end{cor}

For an element $\gc\in\gC$, let us say that a subset $\gS\subset X$ is $\gc$-wandering if all the $\langle\gc\rangle$-translations of $\gS$
are pairwise disjoint or, in other words,  if $\gS\cap\gc^n\cdot\gS=\emptyset$ whenever $n\notin \text{ord}(\gc)\BZ$.
 In particular, a singleton set is $\gc$-wandering iff the corresponding $\langle\gc\rangle$-orbit is faithful. Obviously a subset of a $\gc$-wandering set is also $\gc$-wandering. 

\begin{cor}\label{cor:open-wandering}
Let $\gC\act X$ be a faithful minimal convergence action. Then for every $\gc\in\gC$ there is a $\gc$-wandering open set $O=O(\gc)$ in $X$. 
\end{cor}

\begin{proof}
Suppose first that $\gc$ has infinite order and let $\gO^+,\gO^-\subset X$ be the corresponding sets given in Proposition \ref{prop:main}. By Proposition \ref{prop:main}, $X\setminus (\gO^+\cup\gO^-)$ is $\gc$-wandering and contains an open set.

Next consider the finite order case. By Corollary \ref{cor:wandering}, there is a $\gc$-wandering point $x\in X$. By continuity of the action, if $O$ is a sufficiently small neighbourhood of $x$ then $O$ is $\gc$-wandering.
\end{proof}

\section{The proof of the Theorem \ref{thm}}

Let us reformulate Theorem \ref{thm}:

\begin{thm}\label{thm:torsion}\label{thm:main}
Let $\gC$ be a group admitting a faithful minimal convergence action $\gC\act X$ with $\text{card}(X)> 2$. Then there is a set of representatives $I$ for the non-trivial conjugacy classes of $\gC$, such that $H=\langle I\rangle$ is the free product of the cyclic groups $\{\langle \gc_i\rangle\}_{\gc_i\in I}$ and is strictly contained in $\gC$. 
\end{thm}

\begin{rem}
In fact, for any function 
$$
 F:\{ \text{non-trivial conjugacy classes of}~\gC\}\to \BN\cup\{\infty\},
$$
one can produce an independent set $I_F$ which intersects each non-trivial conjugacy class $C$ exactly $F(C)$ times.
\end{rem}

%


\begin{lem}\label{lem:wandering}
Let $\gC\act X$ be a minimal faithful convergence action. Given an element $\gc\in\gC$ and a proper closed subset $\gS\subset X$, there is an element $\gd=\gd(\gc,\gS)$ for which $\gS$ is $\gc^\gd$-wandering.
\end{lem}

\begin{proof}
Let $O$ be a $\gc$-wandering open set given by Corollary \ref{cor:open-wandering}.
Let $\gd\in\gC$ be an element satisfying $\gd\cdot \gS\subset O$ (see Lemma \ref{lem:placement}). Then $\gS$ is $\gd^{-1}\gc\gd$-wandering.
\end{proof}

The following 
proposition is most suitable in our situation:

\begin{prop}\label{prop:ping-pong}
Suppose that $\ga_n,~n=1,2,\ldots$ are elements of $\gC$ and $\gO_n$ are pairwise disjoint sets, such that $X\setminus \gO_n$ is $\ga_n$-wandering for $n=1,2,\ldots$. Then 
$$
 \gD:=\langle\ga_n,~n=1,2,\ldots\rangle=*_n\langle\ga_n\rangle
$$ 
is the free product of the cyclic groups $\langle\ga_n\rangle,~n=1,2,\ldots$.

Moreover the limit set $L(\gD)$ is contained in $\overline{\cup_n\gO_n}$.
\end{prop} 

\begin{proof}
The proof of the first statement is by a standard ``ping-pong argument". To keep the argument clean let us suppose that $X\setminus\cup_n\gO_n$ is non-empty and contains the point $x$. For $i\in\BN$ and $j\in \BZ$ which is not a multiple of $\text{ord}(\ga_i)$, let $X_i^j=\ga_i^j\cdot(X\setminus\gO_i)$ and observe that by the assumptions of the proposition the sets $X_i^j$ are pairwise disjoint. We claim that for $\ga\in\gD$ we can recover the unique reduced expression of $\ga$ as a word in the $\ga_i$'s by a recursive algorithm based on the position of $\ga\cdot x$. Indeed, if $\ga=\ga_{i_1}^{j_1}\ga_{i_2}^{j_2}\cdots\ga_{i_k}^{j_k}$ then $\ga\cdot x\in X_{i_1}^{j_1}$, and $\ga_{i_1}^{-j_1}\ga\cdot x\in X_{i_2}^{j_2}$, etc.

In order to show the second statement recall that $L(\gD)$ is the set of limits $\lim_n\gc_n\cdot x$ where $(\gc_n)$ runs over all sequences in $\gD$ for which the limit exists. Since for every $\gc\in \gD$, $\gc\cdot x\in \cup_{i,j} X_i^j\subset \cup_i\gO_i$ the statement follows.

\end{proof}

\begin{proof}[The proof of Theorem \ref{thm:torsion}]
Let $\gC\act X$ be a faithful minimal convergence group action. Let $\{C_n\}_{n\in\BN}$ be the non-trivial conjugacy classes of $\gC$, and let $\gO_n,~n\in\BN$ be arbitrary pairwise disjoint open subsets of $X$ whose union is not dense in $X$. In view of Lemma \ref{lem:wandering} we can pick a representative $\gc_n\in C_n$ for each $n$, such that $X\setminus\gO_n$ is $\gc_n$-wandering. By Proposition \ref{prop:ping-pong} the set $\{\gc_n\}$ is independent, i.e. $\gD:=\langle\gc_n:n\in\BN\rangle$ is the free product $*_n\langle \gc_n\rangle$. Finally, since $\cup\gO_n$ is not dense, the limit set of $\gD$ is a proper subset of $X$ and therefore $\gD$ is a proper subgroup of $\gC$. 
\end{proof}

{\it Acknowledgement:} I would like to thank Uri Bader for a valuable remark that enabled removing an unnecessary assumption that appeared in the first manuscript of this paper.

\end{document}